\documentclass[12pt,a4paper,draft]{article}

\usepackage{geometry}
\usepackage[T1]{fontenc}
\usepackage[utf8]{inputenc}
\usepackage{authblk}
\usepackage{amsmath}
\usepackage{amssymb}
\usepackage{amsfonts}
\usepackage{amsthm}
\usepackage{enumerate}
\usepackage{paralist}
\usepackage{amsfonts}
\usepackage{tensor} 
\usepackage{lmodern}
\usepackage{upgreek}
\usepackage{mathrsfs}
\usepackage{nicefrac}
\usepackage{verbatim}
\usepackage{amsmath}
\usepackage[all, cmtip]{xy}\usepackage{amssymb}
\usepackage{amsthm}
\usepackage{extpfeil}
\usepackage{bbm}
\usepackage{paralist}
\usepackage[colorlinks,linkcolor=blue,citecolor=gray,urlcolor=blue,pdftex]{hyperref}

\usepackage[capposition=bottom]{floatrow}

\usepackage[dvipsnames]{pstricks}
\usepackage{pst-solides3d}

\usepackage{enumitem}

\usepackage{titlesec}
\titleformat*{\section}{\large\bfseries}
\titleformat*{\subsection}{\large}

\usepackage[titletoc,toc,title]{appendix}

\usepackage{scalerel}
\usepackage[all]{xy}
\usepackage{tikz}
\usepackage{relsize}
\usepackage{comment}
\usepackage{mathtools}
\usetikzlibrary{decorations.pathmorphing}

\usepackage{mathrsfs}

\usepackage{graphicx}

\usepackage{tkz-euclide}
\usetikzlibrary{intersections}
\usetkzobj{all}
\usetikzlibrary{decorations.pathmorphing}

\tikzset{snake it/.style={decorate, decoration=snake}}

\usepackage{float}

\usepackage{tikz-cd} 

\theoremstyle{plain}
\newtheorem{thm}{Theorem}

\newtheorem{lem}[thm]{Lemma}

\newtheorem{prob}[thm]{Problem}

\newtheorem*{prb*}{Problem}

\newtheorem{claim}{Claim}[thm]

\theoremstyle{definition}
\newtheorem{rem}[thm]{Remark}
\newtheorem{df}[thm]{Definition}
\newtheorem*{ex*}{Example}

\title{A note on Lefschetz spheres and their relatives}
\author{Afshin Goodarzi\thanks{The author is supported by the Knut and Alice Wallenberg Foundation.
}}

\affil{\small{ Royal Institute of Technology, Department of Mathematics, S-100 44, Stockholm, Sweden.} }

\date{ }

\begin{document}

\maketitle

\begin{abstract}
Inspired by the works of Adiprasito, Babson, Nevo, and Murai on the $g$-conjecture, we consider different classes of PL-spheres and the relations between them. We focus on a certain class of spheres that is in the intersection
of vertex-decomposable spheres (a concept due to 
Provan and Billera) and strongly edge-decomposable
spheres (a concept due to Eran Nevo). The spheres in this
class are exactly those vertex-decomposable ones for which Adiprasito's recent proof method in \cite[Theorem 6.3]{Karim} works.

\end{abstract}

\section{Introduction}

Let $\Sigma$ be a $(d-1)$-dimensional simplicial complex. The \emph{$f$-vector} $f(\Sigma)=(f_{-1},f_0,\ldots,f_{d-1})$ of $\Sigma$ lists the number of various dimensional faces in $\Sigma$. The \emph{$h$-vector} $h(\Sigma)=(h_{0},h_1,\ldots,h_{d})$ of $\Sigma$ is defined via the following equations

\begin{equation*}
h_j=\sum_{i=0}^j (-1)^{j-i} {d-i\choose j-i} f_{i-1}\quad \text{ for all } 0\leq j\leq d.
\end{equation*}
\ It is well-known and easy to see that the $f$-vector and the $h$-vector contain the same information. However, many interesting relations can be stated more clearly in terms of the $h$-numbers. Among them is the \emph{Dehn-Sommerville relations} $h_j=h_{d-j}$, $0\leq j\leq d$ for simplicial (homology) spheres \cite[p. 67]{green}. In particular, these relations show that the face numbers of a (homology) sphere $\Sigma$ are completely determined from the first ``half'' of the $h$-vector, or equivalently from the \emph{$g$-vector} $g(\Sigma)=(g_0,g_1,\ldots,g_{\lfloor d/2 \rfloor})$ of $\Sigma$ where $g_0=1$ and $g_j=h_j-h_{j-1}$ for all $1\leq j\leq \lfloor d/2 \rfloor$. 

One of the major open problems on face numbers of simplicial complexes is McMullen's \emph{$g$-conjecture}. It states that a necessary and sufficient condition for a vector $\mathbf{g}=(g_0,g_1,\ldots,g_{\lfloor d/2 \rfloor})$ to be the $g$-vector of a $(d-1)$-dimensional homology sphere is that there exists a set $\mathbf{M}$ of monomials closed under taking divisors, such that for all $1\leq j\leq \lfloor d/2 \rfloor$ the number of monomials of degree $j$ in $\mathbf{M}$ is $g_j$. 

Billera and Lee \cite{BL80} proved the sufficiency of McMullen's $g$-conjecture. The necessity part would follow from an algebraic approach that we briefly describe now. We refer to the book \cite{green} by Stanley for undefined algebraic terminology and to the survey article \cite{Swa} for a more detailed description.

Let $\Sigma$ be a $(d-1)$-dimensional simplicial complex on the vertex set $V=V(\Sigma)$. Let $\mathbb{F}$ be an infinite field and let $R=\mathbb{F}[x_v|v\in V]$ be the polynomial ring with grading induced by setting $\mathrm{deg}\ x_v=1$ for all $v\in V$. The \emph{Stanley-Reisner ring} $A=R/I_\Sigma$ of $\Sigma$ over $\mathbb{F}$ is obtained by dividing $R$ by the ideal $I_\Sigma$ generated by all square-free monomials whose support is not in $\Sigma$. Let $\Sigma$ be Cohen-Macaulay over $\mathbb{F}$ with a symmetric $h$-vector. We say that $\Sigma$ has the \emph{strong Lefschetz property} over $\mathbb{F}$ (or $\Sigma$ is a \emph{Lefschetz complex} over $\mathbb{F}$) if there exist a linear system of parameters $\theta_1,\ldots,\theta_d\in R_1$ and a linear form $\omega\in R_1$ such that the multiplication map 
\[ \omega^{d-2j}: (A/\Theta)_j\rightarrow (A/\Theta)_{d-j}
\]
is a bijection for all $0\leq j\leq \lfloor d/2 \rfloor$, where $\Theta$ is the ideal of $A$ generated by $\theta_1,\ldots,\theta_d$. The \emph{algebraic $g$-conjecture} states that every homology sphere is a Lefschetz complex over $\mathbb{R}$. The algebraic $g$-conjecture implies the $g$-conjecture.

The prototype of this algebraic approach is Stanley's paper \cite{Sta80}, where by using tools from toric geometry he showed that boundary complexes of simplicial polytopes are Lefschetz over $\mathbb{R}$.

In \cite{NevoMinors}, Nevo introduced the class of \emph{strongly edge-decomposable} spheres. It was shown by Babson and Nevo \cite{BabNevo} that strongly edge-decomposable spheres are Lefschetz over any field of characteristic zero. Their result was generalized by Murai, see \cite[Corollary 3.5]{Murai}, where he considers a more general class of strongly edge-decomposable complexes without any assumption on the characteristic of the field. 

 Recently, Adiprasito announced a proof of the algebraic $g$-conjecture for $\mathbb{R}$-homology spheres  \cite{Karim}, and in particular for vertex-decomposable spheres. While the reach and validity of his general claim awaits confirmation, in this note we focus on a class of vertex-decomposable spheres considered in his work \cite[Theorem 6.3]{Karim} (see also Section 6 of his most recent manuscript \cite{Adi19b}) that we call \emph{strongly vertex-decomposable}. Adiprasito's path of argument 
offers a proof of the algebraic $g$-conjecture that is valid  for this class of spheres. We show that strongly vertex-decomposable spheres are in fact strongly edge-decomposable,
and thus, the algebraic $g$-conjecture for this class of spheres follows also from the above-mentioned works by Babson and Nevo \cite{BabNevo} and by Murai \cite{Murai}.

In Figure \ref{Fig1} we illustrate the relations between different classes of spheres and Lefschetz complexes. A solid directed arrow from a class $X$ to the class $Y$ indicates the containment relation $X\subset Y$.

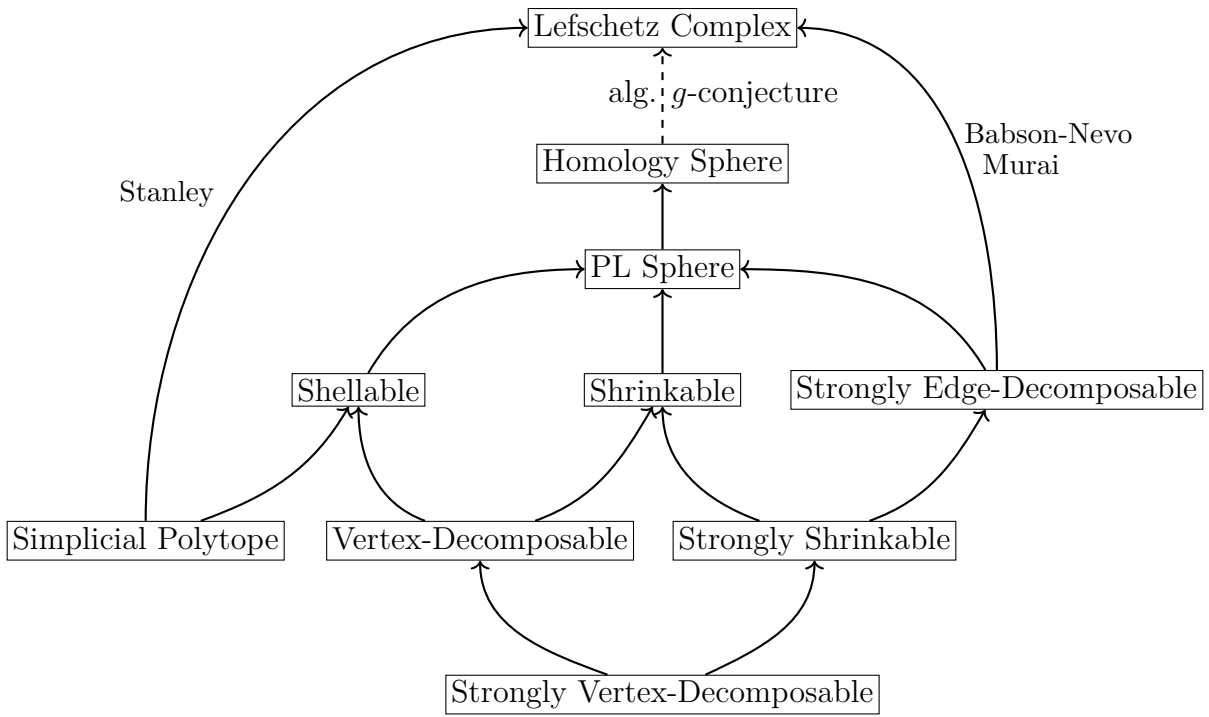
\begin{figure}[H]\label{Fig1}
\begin{tikzpicture}
[scale=.4, vertices/.style={draw, inner sep=2pt}]
%\draw[ball color=black!12, draw] (0,0) circle (1);
%\draw[dashed,thick] (1,0) arc (0:180:1 and .25);
%\draw[thick] (-1,0) arc (180:360:1 and .25);
%\draw[pattern=north east lines] (.5,.5) arc (0:360: .2 and .1);
%\node[vertices] (x) at (1,0){};
\node (y) at (15,15.8){alg. $g$-conjecture};
\node (x) at (24.8,13.5){\small{Murai}};
\node (w) at (25.7,14.5){\small{Babson-Nevo}};

\node (z) at (-3.3,12.5){\small{Stanley}};

%\shadedraw[shading=axis] (3,0) ellipse (.5 and .2);
\node[vertices] (L) at (13,18){Lefschetz Complex};
\node[vertices] (H) at (13,13.5){Homology Sphere};
\node[vertices] (P) at (13,10){PL Sphere};
\node[vertices] (R) at (13,6){Shrinkable};

\node[vertices] (E) at (24,6){Strongly Edge-Decomposable};
\node[vertices] (Sh) at (3,6){Shellable};
\node[vertices] (V) at (7,1){Vertex-Decomposable};

\node[vertices] (Sv) at (13,-4.1){Strongly Vertex-Decomposable};
\node[vertices] (Ss) at (18,1){Strongly Shrinkable};
\node[vertices] (B) at (-4,1){Simplicial Polytope};

\draw[->, dashed,thick] (H)--(L);
\draw[->, thick] (P)--(H);
\draw[->, thick] (E)to [out=90,in=0](L);
\draw[->, thick] (E)to [out=120,in=0](P);
\draw[->, thick] (Sh)to [out=60,in=180](P);
\draw[->, thick] (V)to [out=160,in=-90](Sh);
\draw[->, thick] (R)--(P);
\draw[->, thick] (V)to [out=20,in=-120](R);
\draw[->, thick] (Sv)to [out=160,in=-90](V);
\draw[->, thick] (Sv)to [out=25,in=-90](Ss);
\draw[->, thick] (Ss)to [out=20,in=-120](E);
\draw[->, thick] (Ss)to [out=160,in=-90](R);
\draw[->, thick] (B)to [out=20,in=-120](Sh);
\draw[->, thick] (B) to [out=90,in=-180](L);

%\draw[ thick] (d)--(x);
%\draw[ thick] (y)--(v);

\end{tikzpicture}

\caption{Relations between different classes of spheres}

\end{figure}

\noindent If we disregard the class of Lefschetz complexes in this figure, then $X\subset Y$ if and only if there is a directed path from $X$ to $Y$, except possibly for the case when ``$Y$=strongly edge-decomposable'' and ``$X=$vertex-decomposable'' or ``$X=$shrinkable'' (see Section \ref{Sec:Ex}). 

%%%%%%%%%%%%%%%%  %%%%%%%%%%%%%%%%%%%%%%

\section{Edge Contraction}

We begin by establishing the notation. For undefined terminology and the basic facts about \emph{piecewise linear} (PL) topology, we refer the reader to the chapter \cite{Top} by Bj\"orner. 

Let $\Sigma$ be a simplicial complex and $\sigma$ be a face of $\Sigma$. The \emph{star} $\mathrm{st}(\sigma,\Sigma)$ of $\sigma$ in $\Sigma$ is the set of all faces $\tau\in\Sigma$ such that $\sigma\cup\tau\in\Sigma$. The \emph{antistar} $\mathrm{ast}(\sigma,\Sigma)$ of $\sigma$ in $\Sigma$ is the set of all faces $\tau\in\Sigma$ such that $\sigma\cap\tau=\emptyset$. The \emph{link} of $\sigma$ in $\Sigma$ is defined by \(\mathrm{lk}(\sigma,\Sigma):=\mathrm{st}(\sigma,\Sigma)\cap\mathrm{ast}(\sigma,\Sigma)\). The \emph{deletion} $\Sigma\setminus\sigma$ of $\sigma$ from $\Sigma$ is the subcomplex of $\Sigma$ consisting of all faces that do not contain $\sigma$. 

If $e=uv$ is an edge of $\Sigma$, the \emph{contraction} of $\Sigma$ with respect to $e$ is the simplicial complex $\mathcal{C}(e,\Sigma)$ obtained from $\Sigma$ by identifying $u$ and $v$. More precisely, $\mathcal{C}(e,\Sigma)$ is a simplicial complex on the vertex set $V(\Sigma)\setminus\{u\}$ with the following set of faces 
\[ \mathcal{C}(e,\Sigma)=\{\sigma\in\Sigma | u\notin \sigma\}\cup \{(\sigma\setminus\{u\})\cup\{v\} | u\in\sigma\in\Sigma\}.\]

Let $\Sigma$ be a PL-sphere and $e=uv$ be an edge of $\Sigma$. We say that $\Sigma$ satisfies the \emph{link condition} with respect to $e$ if $\mathrm{lk}(v,\Sigma)\cap\mathrm{lk}(u,\Sigma)=\mathrm{lk}(e,\Sigma)$.

\begin{thm}[Klee-Kleinschmidt and Nevo]\label{contract}
Let $\Sigma$ be a $(d-1)$-dimensional PL-sphere and $e=uv$ be an edge of $\Sigma$. Then the following are equivalent:
\begin{enumerate}
\item[$(i)$] $\mathcal{C}(e,\Sigma)$ is PL-homeomorphic to $\Sigma$,
\item[$(ii)$] $\Sigma$ satisfies the link condition with respect to $e$,
\item[$(iii)$] $(\Sigma\setminus\{u\})\setminus\{v\}$ is a $(d-1)$-dimensional PL-ball,
\item[$(iv)$] $\mathrm{st}(u,\Sigma)\cup\mathrm{st}(v,\Sigma)$ is a $(d-1)$-dimensional PL-ball.
\end{enumerate}

\end{thm}
\begin{proof}
The equivalence between $(i)$, $(ii)$, and $(iii)$ is proved in \cite[6.1]{KK}. To see the equivalence $(iii)\Leftrightarrow (iv)$, let us recall Newman's Theorem (see \cite[12.2]{Top}, for instance) that: if $\Gamma$ is a $(d-1)$-dimensional PL-ball that is a subcomplex of a $(d-1)$-dimensional PL-sphere $\Sigma$, then the closure of the complement of $||\Gamma||$ in $||\Sigma||$ is a $(d-1)$-dimensional PL-ball (here $||.||$ denotes the geometric realization). Now, letting $\Gamma$ to be either of the complexes $(\Sigma\setminus\{u\})\setminus\{v\}$ or $\mathrm{st}(u,\Sigma)\cup\mathrm{st}(v,\Sigma)$, the closure of the complement of $||\Gamma||$ in $||\Sigma||$ is the geometric realization of other one. 
\end{proof}

The equivalence between conditions $(i)$ and $(ii)$ of Theorem \ref{contract} in the more general setting of PL-manifolds (without boundary) is shown in \cite[Theorem 1.4]{NevoMinors}. One can show that the condition $(iv)$ of Theorem \ref{contract} is also equivalent to conditions $(i)$ and $(ii)$ in the more general setting of PL-manifolds.

%One can show the equivalence of the condition $(iv)$ of Theorem \ref{contract} to the first two conditions also in general setting. Below, we include an argument for this. 

%\begin{thm}
%Let $\Sigma$ be a PL-manifold without boundary. Let $e=uv$ be an edge of $\Sigma$. Then $\mathcal{C}(e,\Sigma)$ is PL-homeomorphic to $\Sigma$ if and only if $\mathrm{st}(u,\Sigma)\cup\mathrm{st}(v,\Sigma)$ is a $(d-1)$-dimensional PL-ball.
%\end{thm}
%\begin{proof}
%First assume that $\mathcal{C}(e,\Sigma)$ is PL-homeomorphic to $\Sigma$. Then by \cite[Theorem 1.4]{NevoMinors} $\Sigma$ satisfies the link condition with respect to $e$. It is enough to show that $\mathrm{st}(u,\Sigma)\cap\mathrm{st}(v,\Sigma)$ is a $(d-1)$-dimensional PL-ball. However, one has
%\begin{eqnarray*}
%\mathrm{st}(u,\Sigma)\cap\mathrm{st}(v,\Sigma)&=& \mathrm{st}(e,\Sigma)\cup \left(\mathrm{lk}(u,\Sigma)\cap\mathrm{lk}(v,\Sigma)\right)\\
%&=& \mathrm{st}(e,\Sigma)\cup \mathrm{lk}(e,\Sigma)\\
%&=& \mathrm{st}(e,\Sigma).
%\end{eqnarray*}

%\end{proof}

%%%%%%%%%%%%%%%%  %%%%%%%%%%%%%%%%%%%%%%

\section{Classes of PL-spheres}

%---------------------------------------------------------------

\subsection{Vertex-Decomposable and Shrinkable Spheres}

The concept of vertex-decomposability is a strengthening of the concept of shellability that was introduced by Provan and Billera \cite{PB80} in connection with Hirsch conjecture. 

\begin{df}
A pure $(d-1)$-dimensional simplicial complex $\Sigma$ is said to be \emph{vertex-decomposable} if either $\Sigma$ is a $(d-1)$-simplex or there exists a vertex $v$, called a \emph{shedding vertex}, such that 
\begin{enumerate}
\item $\Sigma\setminus \{v\}$ is $(d-1)$-dimensional vertex-decomposable, and
\item $\mathrm{lk}(v,\Sigma)$ is $(d-2)$-dimensional vertex decomposable.
\end{enumerate}
A \emph{shedding order} $v_1,\ldots,v_k$ of a $(d-1)$-dimensional vertex-decomposable complex $\Sigma$ is a linear ordering of $k=f_0(\Sigma)-d$ of its vertices such that for all $j$ the vertex $v_j$ is a shedding vertex for $(\Sigma\setminus \{v_1\})\setminus\ldots\setminus\{v_{j-1}\}$. Notice that if $v_1,\ldots,v_k$ is a shedding order, then the rest of vertices of $\Sigma$ forms a $(d-1)$-simplex.
\end{df}

Provan and Billera \cite[Theorem 3.1.3]{PB80} proved that $2$-dimensional spheres and balls are vertex-decomposable.
It was shown by Klee and Kleinschmidt \cite[p. 743]{KK} that there is a simplicial $4$-polytope with $10$ vertices whose boundary complex is not vertex-decomposable. 

The following fact about vertex-decomposable balls is well-known. For the sake of completeness we include a proof. 

\begin{lem}\label{lemB}
Let $\Sigma$ be a $(d-1)$-dimensional vertex-decomposable ball and $v$ be a shedding vertex of $\Sigma$. Then 
\begin{enumerate}
\item[$(i)$] $v$ lies on the boundary of $\Sigma$, and
\item[$(ii)$] $\Sigma\setminus\{v\}$ is a $(d-1)$-dimensional PL-ball. 
\end{enumerate}

\end{lem}

\begin{proof}
First notice that for a vertex $u\in\Sigma$ lies on the boundary of $\Sigma$ if and only if $\mathrm{lk}(u,\Sigma)$ is $(d-2)$-dimensional ball and $u\notin\partial \Sigma$ if and only if $\mathrm{lk}(u,\Sigma)$ is $(d-2)$-dimensional sphere. So, in order to prove $(i)$, it suffices to show that for a shedding vertex $v$ one has $\widetilde{H}_{d-2}(\mathrm{lk}(v,\Sigma),\mathbb{Z})=0$. Now, since $\Sigma$ is acyclic and $\Sigma\setminus \{v\}$ is $(d-1)$-connected (see \cite[11.2]{Top}), the Mayer-Vietoris sequence for the pair $\mathrm{st}(v,\Sigma)$ and $\Sigma\setminus\{v\}$ shows that $\widetilde{H}_{d-2}(\mathrm{lk}(v,\Sigma),\mathbb{Z})\cong\widetilde{H}_{d-2}(\Sigma\setminus\{v\},\mathbb{Z})=0$.
This complete the argument for part $(i)$. 

In order to prove $(ii)$, we use induction on $d$. Note that for $d=2$ the statement clearly holds. For a non-empty face $\sigma\in\Sigma\setminus\{v\}$ we observe that if $v\notin\mathrm{lk}(\sigma,\Sigma)$, then $\mathrm{lk}(\sigma,\Sigma)=\mathrm{lk}(\sigma,\Sigma\setminus\{v\})$ is either a PL-ball or a PL-sphere. On the other hand, if $v\in\mathrm{lk}(\sigma,\Sigma)$, then it is not difficult to show that $v$ is a shedding vertex for $\mathrm{lk}(\sigma,\Sigma)$ (see the proof of Proposition 2.3 in \cite{PB80}, for instance). Therefore, $\mathrm{lk}(\sigma,\Sigma)\setminus\{v\}=\mathrm{lk}(\sigma,\Sigma\setminus\{v\})$ must be a PL-ball. In particular, $\Sigma\setminus\{v\}$ is a vertex-decomposable PL-manifold (with non-empty boundary) and hence a PL-ball\cite[Theorem 11.4]{Top}. 

\end{proof}
While any vertex-decomposable sphere contain an edge whose contraction results in a PL-sphere (see \cite[6.2]{KK}, for instance), the class of vertex-decomposable spheres is not closed under edge contractions. By focusing on some essential properties of vertex-decomposable spheres, we extend this family to a family of spheres that behave more nicely under edge contraction.

%---------------------------------------------------------------

\begin{df}
A $(d-1)$-dimensional PL-sphere $\Sigma$ is said to be \emph{shrinkable} if either $\Sigma$ is the boundary of the $d$-simplex or there exists a facet $\sigma$ of $\Sigma$ so that $V(\Sigma)\setminus V(\sigma)$ can be linearly ordered $v_1,v_2,\ldots,v_k$ (called a \emph{shrinking order}) so that for any $1\leq j\leq k$
\begin{enumerate}
\item the subgraph $G_j$ of the $1$-skeleton of $\Sigma$ induced by $v_1,\ldots,v_j$ is connected,
\item The complex \(\mathcal{N}_j(\Sigma):=\bigcup_{i\leq j}\mathrm{st}(v_i,\Sigma)\) is a $(d-1)$-dimensional PL-ball. 
\end{enumerate}
\end{df}

\begin{thm}
Let $\Sigma$ be a $(d-1)$-dimensional vertex-decomposable ball with a shedding order $v_1,\ldots,v_k$. Let $v_0$ be a vertex outside $\Sigma$ and let $\Gamma=(v_0\ast \partial \Sigma)\cup\Sigma$ be the sphere obtined from taking a cone with apex $v_0$ over the boundary of $\Sigma$. Then $\Gamma$ is shrinkable with a shrinking order given by $v_0,v_1,\ldots,v_k$. In particular, every vertex-decomposable sphere is shrinkable and for all $d\geq 5$ there exist $(d-1)$-dimensional shrinkable but not vertex-decomposable spheres. 
\end{thm}
\begin{proof}
 First we observe that part $(ii)$ of Lemma \ref{lemB} implies that for any $0\leq j\leq k$, the complex \[D_j=(\Gamma\setminus\{v_0\})\setminus\{v_1\}\setminus\ldots\setminus\{v_j\}=(\Sigma\setminus\{v_1\})\setminus\ldots\setminus\{v_j\}\] is a $(d-1)$-dimensional vertex-decomposable ball. By Newman's Theorem the closure $X$ of the complement of $||D_j||$ in $||\Gamma||$ is a $(d-1)$-dimensional ball. However, the complex \(\mathcal{N}_j(\Gamma)=\bigcup_{0\leq i\leq j}\mathrm{st}(v_i,\Gamma)\) is a triangulation of $X$. Now, since $v_{j+1}$ is a shedding vertex of $D_j$, it follows from Lemma \ref{lemB} that $v_{j+1}$ has to lie on the boundary of $D_j$ which is also the boundary of $N_j(\Gamma)$. This shows that the subgraph $G_j$ of the $1$-skeleton of $\Sigma$ induced by $v_1,\ldots,v_j$ is connected. Thus $v_0,v_1,\ldots,v_k$ is a shrinking order for $\Gamma$. 
 
 Now observe that if $\Gamma$ is a vertex-decomposable sphere with a shedding vertex $v$, then $\Sigma=\Gamma\setminus\{v\}$ is a vertex-decomposable ball and $\Gamma=(v\ast \partial \Sigma)\cup\Sigma$. Hence, every vertex-decomposable sphere is shrinkable. 
 
Finally, let $\Sigma$ be a PL-sphere whose first barycentric subdivision is not shellable. Such spheres exist in all dimensions greater than or equal $3$. By a result of Pachner \cite[Theorem 5.8]{Pachner}, every PL-sphere can be realized as the boundary complex to a shellable ball. So, we may assume that there exists a shellable ball $\Gamma$ such that $\partial\Gamma=\Sigma$. Let $\mathrm{sd}(\Gamma)$ denote the barycentric subdivision of $\Gamma$. Since the barycentric subdivision of every shellable complex is vertex-decomposable \cite[Corollary 3.3.2]{PB80}, $\mathrm{sd}(\Gamma)$ is a vertex-decomposable ball and the sphere $\Lambda=(w\ast\mathrm{sd}(\Gamma))\cup \mathrm{sd}(\Gamma)$ is shrinkable. However, $\Lambda$ is not vertex-decomposable. Because $\mathrm{lk}(w,\Lambda)= \mathrm{sd}(\Sigma)$ is not vertex-decomposable while the link of every face in a vertex-decomposable complex is vertex-decomposable \cite[Proposition 2.3]{PB80}.

\end{proof}

%---------------------------------------------------------------

\subsection{Strong Versions}
Next, we present the precise definition of the subclass of vertex-decomposable spheres that is considered in Adiprasito's paper (see the paragraph preceding \cite[Theorem 6.3]{Karim}, see also \cite[Section 6]{Adi19b} for a clarified version).
\begin{df}
A $(d-1)$-dimensional vertex-decomposable sphere $\Sigma$ is said to be \emph{strongly vertex-decomposable} if either $d$ is less than or equal $3$, or there exists a shedding order $v_1,\ldots,v_k$ (that we call a \emph{strong shedding order}) such that, for any $1\leq j<k$, the complex $\mathrm{lk}(v_{j+1},\partial \mathcal{N}_j(\Sigma))$ is strongly vertex-decomposable, where \(\mathcal{N}_j(\Sigma):=\bigcup_{i\leq j}\mathrm{st}(v_i,\Sigma)\).
\end{df}

Notice that since every shedding order is a shrinking order, \(\mathcal{N}_j(\Sigma):=\bigcup_{i\leq j}\mathrm{st}(v_i,\Sigma)\) is a $(d-1)$-dimensional ball for all $1\leq j<k$ and it makes sense to talk about $\partial\mathcal{N}_j(\Sigma)$. 

We also consider the strong version of shrinkability. We define the class of \emph{strongly shrinkable} spheres by replacing ``vertex-decomposable'' by ``shrinkable'' in the definition above. Clearly, every strongly vertex-decomposable sphere is strongly shrinkable.

\begin{thm}\emph{(Adiprasito \cite[Theorem 6.3]{Karim})}\label{Kar}\textbf{.}
Let $\Sigma$ be a strongly vertex-decomposable sphere and $\mathbb{F}$ be a field of characteristic zero. Then $\Sigma$ is a Lefschetz complex over $\mathbb{F}$.
\end{thm}

%---------------------------------------------------------------

The following class of PL-spheres is introduced by Nevo \cite{NevoMinors}. 

\begin{df}
A $(d-1)$-dimensional sphere $\Sigma$ is said to be \emph{strongly edge-decomposable} if either $\Sigma$ is the boundary of $d$-simplex or else there exists an edge $e\in\Sigma$ such that $\Sigma$ satisfies the link condition with respect to $e$ and $\mathrm{lk}(e,\Sigma)$ and $\mathcal{C}(e,\Sigma)$ are strongly edge-decomposable. 
\end{df}

All $2$-dimensional spheres are strongly edge-decomposable. Babson and Nevo \cite{BabNevo} proved that strongly edge-decomposable spheres are Lefschetz over any field of characteristic zero. This was generalized by Murai \cite{Murai} where he extended the definition of strong edge-decomposition to complexes that are not necessarily PL-spheres.

\begin{thm}\emph{(Murai \cite[Corollary 3.5]{Murai})}\label{Mur}\textbf{.}
Let $\Sigma$ be a strongly edge-decomposable sphere and $\mathbb{F}$ be an infinite field. Then $\Sigma$ is a Lefschetz complex over $\mathbb{F}$.
\end{thm}

Now, we show that strongly vertex-decomposable spheres are strongly edge-decomposable. 

\begin{thm}\label{main}
Every strongly shrinkable sphere is strongly edge-decomposable. Consequently, every strongly vertex-decomposable sphere is strongly edge-decomposable. 
\end{thm}

\begin{proof}
We use induction both on the dimension and the length of the shrinking order. So, let $\Sigma$ be a $(d-1)$-dimensional shrinkable sphere with a shrinking order $v_1,v_2,\ldots,v_k$. If $d\leq 3$ or $k=1$, then $\Sigma$ is strongly edge-decomposable. Hence, we may assume that $k>1$ and every strongly shrinkable sphere of dimension less than $d-1$ or with a shrinking order of length less than $k$ is strongly edge-decomposable. It follows from the definition that $e=v_1v_2$ is an edge of $\Sigma$ and $\mathrm{st}(v_1,\Sigma)\cup\mathrm{st}(v_2,\Sigma)$ is a $(d-1)$-dimensional ball. Hence, by Theorem \ref{contract}, one sees that $\Sigma$ satisfies the link condition with respect to $e$. In particular, $\mathcal{C}(e,\Sigma)$ is a PL-sphere. Our desired conclusion follows from the claims below: 
\begin{claim}\label{c1}
$\mathcal{C}(e,\Sigma)$ is strongly shrinkable.
\end{claim}
\begin{claim}\label{c2}
$\mathrm{lk}(e,\Sigma)$ is strongly shrinkable.
\end{claim}

\begin{proof}[Proof of Claim \ref{c1}]
Since for any $2\leq j\leq k$ 
\( (\mathcal{C}(e,\Sigma)\setminus\{v_2\})\setminus\ldots\setminus\{v_j\}=(\Sigma\setminus\{v_1\})\setminus\ldots\setminus\{v_j\}\)
is a PL-ball, it follows from the Newman's Theorem that the complex
\[ \bigcup_{2\leq i\leq j}\mathrm{st}(v_i,\mathcal{C}(e,\Sigma))
\]
is also a PL-ball. 
\end{proof}

\begin{proof}[Proof of Claim \ref{c2}]
It follows easily from the definition of strongly shrinkable spheres, the general equation \(\mathrm{lk}(e,\Sigma)=\mathrm{lk}(v_2,\mathrm{lk}(v_1,\Sigma))\) and the fact that \(\mathrm{lk}(v_1,\Sigma)=\partial \mathcal{N}_1(\Sigma)\).
\end{proof}
\end{proof}

\section{Remarks}\label{Sec:Ex}

\begin{rem}
In dimensions less than or equal to $4$ every shedding order is also a strong shedding order and, consequently, all vertex-decomposable spheres of dimension less than or equal to $4$ are strongly vertex-decomposable. In all dimensions greater than $4$, using a result of Pachner \cite[Theorem 5.8]{Pachner}, one can construct examples of shedding orders that are not strong. However, this does not eliminate the possibility of existence of other shedding orders that are strong. 

\end{rem}

It is interesting to construct a concrete example, if such an example exists at all, of a vertex-decomposable sphere which is not strongly vertex-decomposable. Since, every strongly vertex-decomposable sphere is strongly edge-decomposable, it would be enough to find a vertex-decomposable sphere which is not strongly edge-decomposable. 
\begin{prob}
Find a triangulated sphere which is vertex-decomposable but not strongly edge-decomposable. 
\end{prob}

\begin{rem}
For the inductive argument given by Adiprasito in the proof of \cite[Theorem 6.3]{Karim} to work for a PL-sphere $\Sigma$, it is necessary and sufficient that 
\begin{itemize}
\item[(i)] $\Sigma$ is shrinkable with a shrinking order $v_1,v_2,\ldots,v_k$, and 
\item[(ii)] the complexes $\mathrm{lk}(v_{j+1},\partial \mathcal{N}_j(\Sigma))$ are Lefschetz.
\end{itemize}
In particular, the argument given by Adiprasito shows that:
\begin{itemize}
\item[(1)] Every strongly shrinkable sphere is a Lefschetz complex.
\item[(2)] If all PL-spheres of dimension less than or equal to $d$ are Lefschetz complexes, then all shrinkable spheres of dimension less than or equal to $d+2$ are Lefschetz complexes. 
\end{itemize}
See also \cite[Section 6]{Adi19b} where shrinkable and strongly shrinkable spheres are called $B$-decomposable and  $A$-decomposable respectively. 
\end{rem}

\begin{rem}

While we were at the final stage of preparing this note, two new papers by Karu \cite{Kar} and independently by Adiprasito and Steinmeyer \cite{AS} appeared on arXiv announcing simpler proofs of the $g$-conjecture for PL-spheres.
\end{rem}

%---------------------------------------------------------------
%\section*{Acknowledgements}

%The author is supported by the Knut and Alice Wallenberg Foundation.

%%%%%%%%%%%%%%%%%%%%%%%%%%%%%%%%%%%%%%%%%%%%%%%%%
\def\cprime{$'$}
\providecommand{\bysame}{\leavevmode\hbox to3em{\hrulefill}\thinspace}
\providecommand{\MR}{\relax\ifhmode\unskip\space\fi MR }
\providecommand{\MRhref}[2]{%
  \href{http://www.ams.org/mathscinet-getitem?mr=#1}{#2}
}
\providecommand{\href}[2]{#2}

\end{document}